\documentclass[12pt]{article}
\usepackage{graphicx} 
\usepackage{graphicx}%
\usepackage{multirow}%
\usepackage{amsmath,amssymb,amsfonts}%
\usepackage{amsthm}%
\usepackage{mathrsfs}%
\usepackage[title]{appendix}%
\usepackage{xcolor}%
\usepackage{textcomp}%
\usepackage{manyfoot}%
\usepackage{booktabs}%
\usepackage{algorithm}%
\usepackage{algorithmicx}%
\usepackage{algpseudocode}%
\usepackage{listings}%
\usepackage{amssymb,latexsym,amsmath}     
\usepackage{epsf}
\usepackage{array}
\usepackage{amssymb,amscd,comment}
\usepackage{float}
\usepackage{comment, verbatim}
\usepackage[left=2.55cm, right=2.55cm, top=2.55cm, bottom=2.55cm]{geometry}
\usepackage{relsize}
\usepackage{calc}
\usepackage[mathscr]{eucal}
\usepackage{hhline}
\usepackage{tikz}
\usetikzlibrary{positioning}
\usepackage{enumitem}
\setlist[enumerate]{leftmargin=.5in}
\setlist[itemize]{leftmargin=.5in}
\usepackage{multirow}
\usepackage{tabularx}
\usepackage{soul, color}
\usepackage{amssymb,amstext,amsgen,amsfonts,mathrsfs,verbatim,url,hyperref,mathdots, mathtools}
\usepackage{physics}
\usepackage{pgfplots}
\usepackage{orcidlink}
\pgfplotsset{compat=1.18}
\usepackage{mathrsfs}
\usepackage{arydshln}
\ifpdf
  \DeclareGraphicsExtensions{.eps,.pdf,.png,.jpg}
\else
  \DeclareGraphicsExtensions{.eps}
\fi

\newlist{inparaenum}{enumerate}{2}
\setlist[inparaenum]{nosep}
\setlist[inparaenum,1]{label=\bfseries\alph*.}

\def\comment#1{}

\usetikzlibrary{arrows}

\def\invddots{\mathinner{\mskip1mu\raise1pt\vbox{\kern7pt\hbox{.}}\mskip2mu
		\raise4pt\hbox{.}\mskip2mu\raise7pt\hbox{.}\mskip1mu}}

\theoremstyle{plain}%
\newtheorem{theorem}{Theorem}

\theoremstyle{definition}%

\begin{document}

\author{Amin Aboubrahim \orcidlink{0000-0002-1110-4265}\,$^{a}$\footnote{\href{mailto:abouibrah@hartford.edu}{abouibrah@hartford.edu}}~,~Anastasiia Minenkova \orcidlink{0000-0002-8233-7971}\,$^b$\footnote{\href{mailto:minenkova@hartford.edu}{minenkova@hartford.edu}}~~and Tomas Yonkers \orcidlink{0009-0004-4190-6340}\,$^{a,b}$\footnote{\href{mailto:yonkers@hartford.edu}{yonkers@hartford.edu}} \\~\\
$^{a}$\textit{\normalsize Department of Physics, University of Hartford,} \\
\textit{\normalsize 200 Bloomfield Ave., West Hartford, CT 06117, U.S.A.} \\
$^{b}$\textit{\normalsize Department of Mathematics, University of Hartford,} \\
\textit{\normalsize 200 Bloomfield Ave., West Hartford, CT 06117, U.S.A.}
}

\title{
\Large \bf
Hypergeometric solutions to coupled Klein-Gordon equations in cosmology with small-field excursions
}

\date{}

\maketitle

\begin{abstract}

The Klein-Gordon equation is a second order differential equation that appears in a variety of applications in mathematics, physics and engineering. This equation is often solved using numerical methods, since a general analytical solution is unattainable. However, one can find very useful analytical solutions for some special cases. In cosmology, the Klein-Gordon equation of a field can explain the physics of inflation and has important consequences on understanding the origin and evolution of the universe. An analytical solution to the Klein-Gordon equation in cosmology has been obtained in the small-field excursion approximation, and it turns out that the solution can be written in terms of the confluent hypergeometric limit function. In this work, we derive this result again while showing the details of the calculations and then generalize to the case of two coupled Klein-Gordon equations. We show that the solution of this coupled system can be written in terms of Kummer's hypergeometric limit function while considering minimal assumptions. We carry out a full numerical calculation of the coupled system and compare it to the results from the analytical expressions.

\end{abstract}

\numberwithin{equation}{section}
\newpage

{  \hrule height 0.4mm \hypersetup{linkcolor=black,linktocpage=true} \tableofcontents
\vspace{0.5cm}
 \hrule height 0.4mm}

\section{Introduction}  

Orthogonal polynomials and special functions play a fundamental role in applied mathematics because they provide exact solutions to many differential equations arising in physics and engineering. Classical families such as Legendre, Hermite, and Laguerre polynomials, and Bessel functions appear in problems involving wave propagation, quantum mechanics, electromagnetism, and fluid dynamics. 

This study focuses on the Klein-Gordon (KG) equation, which is a second order ordinary differential equation that appears in many applications of mathematics and physics. The KG equation is solved numerically since exact solutions are rare because of the complexity of the underlying equations.
In cosmology, the KG equation describes the evolution of a spin zero field (a scalar or pseudoscalar field) which can act as dark matter or dark energy~\cite{Aboubrahim:2024spa,Aboubrahim:2024cyk,Aboubrahim:2025usl}. A full and general analytical solution is not possible because the expanding universe introduces a time-dependent ``friction'' term and the choice of the potential term can impose different time scales in the equation. However, a full or partial analytical solution is possible for particular potential terms and for small field excursions~\cite{Poulin:2018dzj,Tsujikawa:2013fta}. In~\cite{Poulin:2018dzj}, the authors analytically solved the KG equation for early time evolution before the field enters into an oscillatory behavior. The solution turns out to be the confluent hypergeometric limit function ${}_0F_1(;b;z)$.

The importance of these mathematical structures extends beyond the aforementioned KG equations. Orthogonal polynomials and special functions frequently arise in the construction of spectral methods, approximation schemes, and eigenfunction expansions for differential equations. Namely, hypergeometric functions provide a unified framework for studying the behavior of solutions, deriving asymptotic approximations, and identifying relationships among seemingly distinct physical systems.

We begin with a review of confluent hypergeometric functions and their connections to generalized Laguerre polynomials. We then consider the KG equation describing a dark matter scalar field and show that its linearized early-time evolution can be expressed in terms of the confluent hypergeometric limit function ${}_0F_1$. Next, we extend the analysis to a coupled system of dark matter and dark energy fields and demonstrate that the resulting equations can be reduced to Kummer's differential equation, yielding solutions involving ${}_1F_1$. Finally, we compare the resulting analytical expressions with numerical computations and investigate their behavior across a range of coupling strengths.

\section{The confluent hypergeometric functions}

The confluent hypergeometric functions form an important class of special functions and emerge naturally as confluent limits of the generalized hypergeometric functions.  Particularly important are the confluent hypergeometric limit function ${}_0F_1(;b;z)$ and Kummer's confluent hypergeometric function of the first kind ${}_1F_1(a;b;z)$ (see~\cite{Jr:2021rnt} for a review).

\subsection{Definition and the differential equation}

Kummer's differential equation is
\begin{equation}
z\frac{\dd^2f}{\dd z^2}+(b-z)\frac{\dd f}{\dd z}-af=0\,,
\label{eq:kummer}
\end{equation}
where $a,b\in\mathbb C$ are fixed parameters. The point $z=0$ is a regular singular point, whereas $z=\infty$ is an irregular singular point. Equation \eqref{eq:kummer} has two linearly independent
solutions, one of which is Kummer's confluent hypergeometric function of the first kind, introduced by Kummer in 1837 \cite{Kummer1837}.

The function ${}_1F_1(a;b;z)$ is defined by the absolutely convergent power series
\begin{equation}
{}_1F_1(a;b;z)=\sum_{n=0}^{\infty}\frac{(a)_n}{(b)_n}\frac{z^n}{n!}\,,
\label{eq:1f1series}
\end{equation}
where
\[
(a)_n=a(a+1)\cdots(a+n-1)=\frac{\Gamma(a+n)}{\Gamma(a)}
\]
is the Pochhammer symbol or rising factorial, with $(a)_0=1$. Provided that $b\notin\mathbb{Z}_{\le0}$, the series converges for all finite values of $z$, and hence ${}_1F_1(a;b;z)$ is an entire function of $z$.

A second  confluent hypergeometric function we consder is the confluent hypergeometric limit function
\begin{equation}
{}_0F_1(;b;z)=\sum_{n=0}^{\infty}\frac{z^n}{(b)_n n!}\,.
\label{eq:0f1series}
\end{equation}
This function may be obtained as a confluent limit of ${}_1F_1(a;b;z)$:
\begin{equation*}
{}_0F_1(;b;z)=\lim_{a\rightarrow\infty}{}_1F_1\left(a;b;\frac{z}{a}\right)\,.
\end{equation*}
The function ${}_0F_1$ satisfies the differential equation
\begin{equation}
z\frac{\dd^2 f}{\dd z^2}+b\frac{\dd f}{\dd z}-f=0\,.
\label{eq:0f1ode}
\end{equation}
Because of its close relationship to Bessel functions, ${}_0F_1$ serves as an important link between hypergeometric and Bessel function theory.

\subsection{Basic properties}

The confluent hypergeometric function satisfies
\[
{}_1F_1(a;b;0)=1.
\]
Differentiating term-by-term in the defining series yields the fundamental identity
\begin{equation*}
\frac{\dd}{\dd z}\,{}_1F_1(a;b;z)=\frac{a}{b}\,{}_1F_1(a+1;b+1;z).
\end{equation*}
One of the key identities satisfied by ${}_1F_1(a;b;z)$ is Kummer's transformation:
\begin{equation}
{}_1F_1(a;b;z)=e^{z}\,{}_1F_1(b-a;b;-z).
\label{eq:kummertransform}
\end{equation}
This transformation is frequently used in asymptotic analysis and in the derivation of numerous identities involving confluent hypergeometric functions.

The confluent hypergeometric limit function ${}_0F_1$ is closely related to the Bessel function of the first kind. Specifically,
\begin{equation}
J_\nu(x)=\frac{\left(\frac{x}{2}\right)^\nu}{\Gamma(\nu+1)}{}_0F_1\left(;\nu+1;-\frac{x^2}{4}\right).
\label{eq:besselrep}
\end{equation}
Kummer’s confluent hypergeometric function of the first kind ${}_1F_1(a;b;z)$ can also be related to the Bessel function of the first kind through
\begin{equation}
    {}_1F_1\left(\frac{a}{2};a;ibx\right)=e^{ibx/2}\,\Gamma\left(\frac{a+1}{2}\right)\left(\frac{bx}{4}\right)^{\frac{1-a}{2}}J_{\frac{a-1}{2}}\left(\frac{bx}{2}\right)\,,
    \label{eq:f11bessel}
\end{equation}
for $a,b,x>0$~\cite{parida:hal-05236085}. Consequently, many properties of Bessel functions can be interpreted as properties of the hypergeometric functions ${}_0F_1$ and ${}_1F_1$.

\subsection{Asymptotics}

The asymptotic behavior of confluent hypergeometric functions is considerably richer than that of most elementary functions.

For $|\arg z|<3\pi/2$, Kummer's function admits the asymptotic expansion
\begin{equation*}
{}_1F_1(a;b;z)\sim\frac{\Gamma(b)}{\Gamma(a)}e^{z}z^{a-b}\left[1+\frac{(b-a)(1-a)}{z}+\mathcal{O}(z^{-2})\right], \qquad z\rightarrow\infty\,.
\end{equation*}
This expression shows that ${}_1F_1(a;b;z)$ generally grows exponentially as $z$ becomes large.

The second independent solution of Kummer's equation~\eqref{eq:kummer}, usually denoted by $U(a,b,z)$ and referred to as Tricomi's function, has the asymptotic behavior
\begin{equation}
U(a,b,z)\sim z^{-a}, \qquad |z|\rightarrow\infty\,,    
\label{eq:tricomi}
\end{equation}
and is therefore of particular interest in boundary-value problems requiring solutions that decay at infinity.

For the limit function ${}_0F_1$, asymptotic formulas may be derived from its connection with Bessel functions. In particular,
\[
{}_0F_1(;b;z)\sim\frac{\Gamma(b)}{2\sqrt{\pi}}z^{\frac14-\frac b2}e^{2\sqrt z},\qquad
z\rightarrow+\infty\,.
\]
This behavior resembles the asymptotic growth of modified Bessel functions.

\subsection{Particular examples}

Many elementary and special functions arise as special cases of confluent hypergeometric functions.

\paragraph{The exponential function.}

When $a=b$, \eqref{eq:1f1series} reduces to
\[
{}_1F_1(a;a;z)=e^z\,.
\]
Thus the exponential function itself is a special case of a confluent hypergeometric function.

\paragraph{Bessel functions.}

Using \eqref{eq:besselrep}, we write
\[
{}_0F_1\left(;\nu+1;-\frac{x^2}{4}\right)=\Gamma(\nu+1)\left(\frac{x}{2}\right)^{-\nu} J_\nu(x)\,.
\]
Hence Bessel functions can also be viewed as special cases of generalized hypergeometric functions.

\paragraph{The error function.}

The error function admits the representation
\begin{equation*}
\operatorname{erf}(x)=\frac{2x}{\sqrt{\pi}}{}_1F_1\left(\frac12;\frac32;-x^2\right)\,,
\end{equation*}
which explains the frequent appearance of confluent hypergeometric functions in diffusion theory and probability.

\paragraph{The gamma function.}

The gamma function $\Gamma(a)$ is defined as
\begin{equation}
    \Gamma(a)=\int_0^\infty t^{a-1}e^{-t}\,\dd t\,.
    \label{eq:gamma}
\end{equation}
It can also be written in terms of the confluent hypergeometric function~\cite{Sneddon_1961}
\begin{equation*}
    \Gamma(a)=\lim_{z\to\infty}\frac{z^a}{a}\,{}_1F_1(a;a+1;-z)\,.
\end{equation*}

\paragraph{Laguerre polynomials.}

One of the most significant polynomial connections to Kummer's
function occurs when
\[
a=-n, \qquad n\in\mathbb N_0\,.
\]
Since
\[(-n)_k=0, \qquad k>n\,,
\]
the infinite series terminates and becomes a polynomial of degree $n$. The generalized Laguerre polynomials are defined by
\begin{equation}
L_n^{(\alpha)}(z)=\frac{(\alpha+1)_n}{n!}{}_1F_1(-n;\alpha+1;z)\,.
\label{eq:laguerrehyper}
\end{equation}
The first few generalized Laguerre polynomials are
\begin{align*}
L_0^{(\alpha)}(z)&=1\,, \\
L_1^{(\alpha)}(z)&=\alpha+1-z\,,
\end{align*}
and
\[
L_2^{(\alpha)}(z)=\frac12\left[z^2-2(\alpha+2)z+(\alpha+1)(\alpha+2)\right]\,.
\]
Substituting $a=-n$ and $b=\alpha+1$ into Kummer's equation~\eqref{eq:kummer} produces the Laguerre differential equation
\begin{equation}
z\frac{\dd^2y}{\dd z^2}+(\alpha+1-z)\frac{\dd y}{\dd z}+ny=0\,.
\label{eq:laguerreode}
\end{equation}
The generalized Laguerre polynomials form an orthogonal family on $[0,\infty)$ with respect to the weight function
\[
w(z)=z^\alpha e^{-z}\,,
\]
satisfying
\begin{equation*}
\int_0^\infty e^{-z}z^\alpha L_m^{(\alpha)}(z) L_n^{(\alpha)}(z)\,\dd z=
\frac{\Gamma(n+\alpha+1)}{n!}\delta_{mn}\,,
\end{equation*}
where $\delta_{mn}$ is the Kronecker delta.

\paragraph{Applications of Laguerre polynomials.}

Laguerre polynomials play a central role in mathematical physics. Their most famous application occurs in quantum mechanics, where they appear in the exact solution of the hydrogen atom (e.g. see~\cite{zettili2009quantum}). After separation of variables in spherical coordinates, the radial wavefunctions are given by
\[
R_{nl}(r)=C_{nl}e^{-r/(na_0)}\left(\frac{2r}{na_0}\right)^l L_{n-l-1}^{(2l+1)} \left(\frac{2r}{na_0}\right),
\]
where $a_0$ denotes the Bohr radius and $C_{nl}$ is a normalization constant. Consequently, the discrete energy levels and orbital structures of hydrogen-like atoms are intimately connected with generalized Laguerre polynomials.

In numerical analysis, Laguerre polynomials form the basis of Gauss--Laguerre quadrature, a highly efficient technique for evaluating integrals of the form
\[
\int_0^\infty e^{-x}f(x)\,\dd x.
\]
Finally, Laguerre polynomials appear in probability theory and statistics (e.g. see~\cite{Lawi2008} and \cite{MUSTAPHA20102178}).

Therefore, Laguerre polynomials provide a remarkable connection between confluent hypergeometric functions, orthogonal polynomial theory, quantum mechanics, optics, probability theory, and numerical analysis.

\section{The Klein-Gordon equation in cosmology}

We give an application of the confluent hypergeometric function obtained as a solution to the KG equation in cosmology. We begin with a single KG equation and then consider the case of two coupled KG equations representing dark matter and dark energy fields. 

\subsection{The single Klein-Gordon equation for a dark matter field}

In this section, we follow closely the implementation of~\cite{Poulin:2018dzj}. In their work, the authors rewrite the KG equation in terms of dimensionless quantities and directly show their final result. This section expands on the omitted steps and presents the progression leading to the final result.

The spacetime interval for an isotropic, homogeneous and spatially-flat universe is given by the Friedmann–Lema\^{i}tre–Robertson–Walker (FLRW) metric
\begin{equation*}
    \dd s^2=-\dd t^2+a^2(t)\sum_{1\le j,k \le3}\delta_{jk}\dd x_j \dd x_k\,,
\end{equation*}
where $a(t)$ is the scale factor, $t$ denotes the cosmic time and the symbol $\delta_{jk}$ is the Kronecker delta.  

The evolution of the scalar field $\phi(t)$ (here represents dark matter) is described by the Klein-Gordon equation
\begin{equation*}
    \ddot{\phi}+3H\dot{\phi}+\frac{\dd V_n(\phi)}{\dd\phi}=0\,,
\end{equation*}
where a dot denotes a derivative with respect to cosmic time, $H=\dot{a}/a$ is the Hubble parameter and the potential is given by
\begin{equation*}
V_n(\phi)=\Lambda^4\left(1-\cos\frac{\phi}{f}\right)^n\,.    
\end{equation*}
The quantities $\Lambda$ and $f$ are real and positive while $n$ is a positive integer. The Hubble parameter can also be expressed as $H=H_0 E(a)$, where $H_0$ is today's Hubble parameter (a constant) and $E(a)$ is a function of the scale factor. In order to solve the KG equation, it is convenient to redefine the variables so that they are dimensionless. If we define $\Theta\equiv\phi/f$, $m\equiv \Lambda^2/f$, $\alpha\equiv f/M_p$~(here $M_p$ is the Planck mass), $x\equiv H_0 t$ and $\mu\equiv m/H_0$, then the ordinary second order differential equation we have to solve takes the form
\begin{equation}\label{eq:diffeqTheta}
\Theta'' + \frac{3p}{x}\,\Theta' + \alpha^{-2} \frac{\dd V_n}{\dd\Theta}(\Theta) = 0,   
\end{equation}
where the prime is a derivative with respect to $x$. Based on the new variables, the potential now takes the form
\begin{align}\label{eq:vntheta}
V_n(\Theta) &= \mu^2 \alpha^2 \bigl(1 - \cos\Theta\bigr)^n\,,
\end{align}
and
\begin{align}
    &H=\frac{\dot{a}}{a}=\frac{1}{a}\frac{\dd a}{\dd t}=\frac{1}{a}\frac{\dd a}{\dd x}\frac{\dd x}{\dd t}=  \frac{a^\prime}{a}H_0\;\text{ or }\; \frac{H}{H_0}=\frac{a^\prime}{a}=E(a)\,.
\end{align}
So, we have
\begin{equation*}
    E(a)=\frac{a^\prime}{a}=\frac{p}{x}\,,
\end{equation*}
with $p=1/2$ during radiation domination and $p=2/3$ during matter domination.

We expand the potential $V_n(\Theta)$ about its initial value $\Theta_i$:
\begin{align}
V_n(\Theta)= V_n(\Theta_i) + \frac{\dd V_n}{\dd\Theta} (\Theta_i)\,\Big(\Theta-\Theta_i\Big)+ \frac{1}{2} \frac{\dd^2V_n}{\dd\Theta^2}(\Theta_i)\,\Big(\Theta-\Theta_i\Big)^2 + \cdots  
\label{Vn}
\end{align}
Taking the derivative of \eqref{Vn}, we get
\begin{align}
\frac{\dd V_n}{\dd\Theta}(\Theta)&=\frac{\dd V_n}{\dd\Theta} (\Theta_i)+ \frac{\dd^2V_n}{\dd\Theta^2}(\Theta_i)\,\Big(\Theta-\Theta_i\Big) + \cdots \nonumber \\
&=\frac{\dd V_n}{\dd\Theta} (\Theta_i)+ \frac{\dd^2V_n}{\dd\Theta^2}(\Theta_i)\,\delta\Theta+ \mathcal{O}(\delta\Theta^2),
\end{align}
where $\delta\Theta=\Theta-\Theta_i$ such that $\delta\Theta\ll \Theta_i$ so that $\mathcal{O}(\delta\Theta^2)$ and higher order terms can be dropped.

Now, let us compute the derivatives using \eqref{eq:vntheta}:
\begin{align}
\frac{\dd V_n}{\dd\Theta}(\Theta) &= \mu^2\alpha^2\,n\,(1-\cos\Theta)^{n-1}\sin\Theta, \\
\frac{\dd^2V_n}{\dd\Theta^2}(\Theta)&= \mu^2\alpha^2\,n\,(n-1)\,(1-\cos\Theta)^{n-2}\sin^2\Theta+\mu^2\alpha^2\,n\,(1-\cos\Theta)^{n-1}\cos\Theta \nonumber \\
&=\mu^2\alpha^2\,n\,(1-\cos\Theta)^{n-1}(n+n\cos\Theta-1)\,.
\end{align}
From this we can see that
\begin{equation*}
\frac{\dd V_n}{\dd\Theta}(\Theta_i)=\left(\frac{\sin\Theta_i}{n+n\cos\Theta_i-1}\right)\frac{\dd^2V_n}{\dd\Theta^2}(\Theta_i) \,.   
\end{equation*}
So, we then write
\begin{equation*}
\frac{\dd V_n}{\dd\Theta}(\Theta)=\left(\frac{\sin\Theta_i}{n+n\cos\Theta_i-1}+\delta\Theta\right)\frac{\dd^2V_n}{\dd\Theta^2}(\Theta_i)\,.
\end{equation*}
Therefore, the differential equation \eqref{eq:diffeqTheta} becomes
\begin{equation*}
\delta\Theta'' + \frac{3p}{x}\,\delta\Theta' + \left(\frac{\sin\Theta_i}{n+n\cos\Theta_i-1}+\delta\Theta\right)\alpha^{-2}\frac{\dd^2V_n}{\dd\Theta^2}(\Theta_i)=0\,.    
\end{equation*}
Setting $\beta(\Theta_i)\equiv \sin\Theta_i/(n+n\cos\Theta_i-1)$, $y\equiv \beta(\Theta_i)+\delta\Theta$ and
\begin{equation*}
    M^2\equiv \alpha^{-2}\frac{\dd^2V_n}{\dd\Theta^2}(\Theta_i)=\mu^2\,n\,(1-\cos\Theta_i)^{n-1}(n+n\cos\Theta_i-1)\,,
\end{equation*}
we finally have
\begin{equation}
y'' + \frac{3p}{x}y' + M^2  y= 0.
\label{dfq1}
\end{equation}
The confluent hypergeometric limit function ${}_0F_1(;b;z)$ satisfies the differential equation~(\ref{eq:0f1ode}).
To bring \eqref{dfq1} to the form of \eqref{eq:0f1ode}, define
\begin{equation*}
z = \xi x^n, \qquad y(x) = f(z(x))\,,    
\end{equation*}
where $\xi$ and $n$ are constants to be determined. Then
\begin{equation*}
\frac{\dd y}{\dd x} = \frac{\dd f}{\dd z}\frac{\dd z}{\dd x} = \xi n x^{n-1} \frac{\dd f}{\dd z}\,,
\end{equation*}
and
\begin{align}
\frac{\dd^2y}{\dd x^2} &=\frac{\dd}{\dd x}\left(\frac{\dd f}{\dd z}\frac{\dd z}{\dd x}\right)=\frac{\dd}{\dd z}\left(\frac{\dd f}{\dd z}\frac{\dd z}{\dd x}\right)\frac{\dd z}{\dd x}  \nonumber \\
&=\left(\frac{\dd^2 f}{\dd z^2}\frac{\dd z}{\dd x}\right)\frac{\dd z}{\dd x}+\left(\frac{\dd f}{\dd z}\frac{\dd}{\dd z}\frac{\dd z}{\dd x}\right)\frac{\dd z}{\dd x} \nonumber \\
&=\frac{\dd^2 f}{\dd z^2}\left(\frac{\dd z}{\dd x}\right)^2 + \frac{\dd f}{\dd z}\frac{\dd^2 z}{\dd x^2} \nonumber \\
&= n^2\xi^2 x^{2n-2} \frac{\dd^2 f}{\dd z^2} + \xi n(n-1)x^{n-2} \frac{\dd f}{\dd z}\,.    
\end{align}
Substitute these into \eqref{dfq1}:
\begin{align}
0
&= y'' + \frac{3p}{x}y' + M^2 y \nonumber \\
&= \left[n^2\xi^2 x^{2n-2} \frac{\dd^2 f}{\dd z^2} + \xi n(n-1)x^{n-2} \frac{\dd f}{\dd z}\right]
 + \frac{3p}{x}\left(\xi n x^{n-1} \frac{\dd f}{\dd z}\right)
 + M^2 f(z) \nonumber \\
&=n^2\xi^2 x^{2n-2} \frac{\dd^2 f}{\dd z^2}+\left[\xi n(n-1)x^{n-2}+\frac{3p}{x}\xi n x^{n-1}\right]\frac{\dd f}{\dd z}+M^2 f(z)\,.   
\end{align}
Multiply both sides by $\dfrac{x^{2-n}}{n^2 \xi}$ to get:
\begin{equation}
\xi x^n \frac{\dd^2 f}{\dd z^2}+\left(\frac{n-1}{n}+\frac{3p}{n}\right)\frac{\dd f}{\dd z}+\frac{M^2}{n^2 \xi}x^{2-n}f(z)=0\,.
\label{eqxy}
\end{equation}
For the last term in \eqref{eqxy} to have a  constant coefficient one must have $n=2$. We then get
\begin{equation*}
    \xi x^2 \frac{\dd^2 f}{\dd z^2}+\left(\frac{1+3p}{2}\right)\frac{\dd f}{\dd z}+\frac{M^2}{4\xi}f(z)=0\,.
\end{equation*}
Finally, setting $\xi\equiv -M^2/4$ will give
\begin{equation}
z \frac{\dd^2 f}{\dd z^2} + \left(\frac{1+3p}{2}\right) \frac{\dd f}{\dd z} - f(z) = 0.
\label{dfq2}
\end{equation}
Comparing \eqref{dfq2} with \eqref{eq:0f1ode}, we identify
\begin{equation*}
b \equiv \frac{1+3p}{2}\,.    
\end{equation*}
The regular solution of \eqref{dfq2} is the confluent hypergeometric limit function~(\ref{eq:0f1series}):
\begin{equation*}
f(z) = {}_0F_1\!\left(; \frac{1+3p}{2}; z\right).    
\end{equation*}
Knowing that $z = -M^2 x^2 / 4$, we find
\begin{equation*}
y(x) = C\,\,{}_0F_1\!\left(; \frac{1+3p}{2}; -\frac{M^2 x^2}{4}\right)\,,    
\end{equation*}
where $C$ is a constant of integration to be determined. Restoring $\Theta = \Theta_i + y-\beta$, we have
\begin{equation*}
\Theta(x) = \Theta_i + C\,\,{}_0F_1\!\left(; \frac{1+3p}{2}; -\frac{M^2\,x^2}{4}\right)-\beta(\Theta_i)\,.    
\end{equation*}
As $x\to 0$, $\Theta$ approaches $\Theta_i$, and since ${}_0F_1(;a;0)=1$, we find that $C=\beta(\Theta_i)$. So, we finally write
\begin{equation}
\Theta(x) = \Theta_i + \frac{\sin\Theta_i}{n+n\cos\Theta_i-1}\Bigg[{}_0F_1\!\left(; \frac{1+3p}{2}; -\frac{M^2\,x^2}{4}\right)-1\Bigg]\,.
\label{Tsol}
\end{equation}
Set $\mathcal{A}\equiv -\frac{1}{4}M^2$. Let us expand the confluent hypergeometric limiting function:
\begin{align}
{}_0F_1\!\left(; \frac{1+3p}{2}; \mathcal{A}\,x^2\right)&=\sum_{k=0}^{\infty}\frac{(\mathcal{A}x^2)^k}{\left(\frac{3p+1}{2}\right)_k\,k!}\simeq 1+\frac{\mathcal{A}x^2}{\frac{3p+1}{2}1!}+\mathcal{O}(\mathcal{A}^2x^4) \nonumber \\
&=1+\frac{2\mathcal{A}x^2}{3p+1}+\mathcal{O}(\mathcal{A}^2x^4)=1-\frac{M^2x^2}{2(3p+1)}+\mathcal{O}(\mathcal{A}^2x^4)\,.
\end{align}
Replacing this expression back into \eqref{Tsol}, we get
\begin{align}
\Theta(x) &\simeq \Theta_i + \frac{\sin\Theta_i}{n+n\cos\Theta_i-1}\Bigg[1-\frac{M^2x^2}{2(3p+1)}+\mathcal{O}(\mathcal{A}^2x^4)-1\Bigg] \nonumber \\
&\simeq \Theta_i - \frac{\sin\Theta_i}{n+n\cos\Theta_i-1}\mu^2\,n\,(1-\cos\Theta_i)^{n-1}(n+n\cos\Theta_i-1)\frac{x^2}{2(3p+1)}+\mathcal{O}(\mathcal{A}^2x^4) \nonumber \\
&\simeq \Theta_i-\frac{\mu^2 n x^2\sin\Theta_i(1-\cos\Theta_i)^{n-1}}{2(3p+1)}+\mathcal{O}(\mathcal{A}^2x^4)\,,
\end{align}
which is precisely the result shown in~\cite{Poulin:2018dzj}.

\subsection{Coupled Klein-Gordon equations: interacting dark matter and dark energy fields}

The preceding section demonstrated that the linearized KG equation for a single dark matter scalar field can be transformed into a differential equation whose regular solution is the confluent hypergeometric limit function ${}_0F_1$. While this provides a useful analytical description of single-field dynamics, many cosmological models allow interactions between different dark-sector components. In particular, coupling dark matter and dark energy through interacting scalar fields can significantly alter the evolution of the Universe. We therefore extend the analysis to a system of coupled KG equations and investigate whether exact analytical solutions can still be obtained using the theory of special functions. 

In cosmology, a system of coupled Klein-Gordon equations representing two scalar fields~\cite{Aboubrahim:2026gyp,Aboubrahim:2025usl,Aboubrahim:2024spa,Aboubrahim:2024cyk,Rahimy:2025iyj,Teixeira:2024qmw} has been used in an attempt to explain the different anomalies observed today. These anomalies represent differences between theory predictions and experimental measurements~\cite{CosmoVerseNetwork:2025alb,Abdalla:2022yfr}. It is therefore of interest to consider this system of interacting (coupled) KG equations to better understand the dynamics of the two fields as they evolve with time.   

Consider two scalar fields $\chi$ and $\phi$ whose coupled Klein-Gordon equations are given by (in cosmic time):
\begin{align}
\begin{split}
        &\ddot{\chi}+3H\dot{\chi}+\frac{\dd V_n(\chi)}{\dd\chi}=-\frac{Q}{\dot{\chi}}\,, \\
    &\ddot{\phi}+3H\dot{\phi}+\frac{\dd V(\phi)}{\dd\phi}=\frac{Q}{\dot{\phi}}\,,
    \end{split}
    \label{eq:CoupledKG}
\end{align}
where the potentials are given by
\begin{align}
    \label{p1}
    V_n(\chi)&=\Lambda_\chi^4\Big[1-\cos(\chi/F)\Big]^n\,, \\
    V(\phi)&=\Lambda_\phi^4\Big[1+\cos(\phi/f)\Big]\,.
    \label{p2}
\end{align}
Again, $\Lambda_\chi$, $\Lambda_\phi$, $F$ and $f$ are real and positive constants. 

\begin{theorem}[Hypergeometric solution of the coupled Klein--Gordon system] Let  $ M_{\phi}^{2} = M_{\chi}^{2} \equiv M^{2}$  and  $\gamma_{\phi} = \gamma_{\chi} \equiv \gamma$. Then  the solution to the coupled Klein--Gordon system \eqref{eq:CoupledKG} is
\begin{align} 
\begin{split}
    X(x) &= X_i - \beta(X_i) + \operatorname{Im} \bigl( u(x) \bigr), \\
    \Phi(x) &= \Phi_i - \tan(\Phi_i) + \operatorname{Re} \bigl( u(x) \bigr),
\end{split} \end{align} 
where
\begin{equation*}
    u(x) = C e^{-i(k+\gamma/2)x}{}_1F_1 \left( A; 3p; 2ikx \right)\,,
\end{equation*}
with $ k = \sqrt{ M^{2} + \dfrac{\gamma^{2}}{4} }$ and $A = \dfrac{3p}{2k} \left( k+\dfrac{\gamma}{2} \right)$. Here, $X_i$ and $\Phi_i$ represents the initial values of the fields $X$ and $\Phi$.
\end{theorem}

\begin{proof}
We use a source term of the form $Q=\lambda\dot{\chi}\dot{\phi}$~\cite{Garcia-Arroyo:2024tqq} which implies that the KG equations now become
\begin{align}
\label{kg1}
    &\ddot{\chi}+3H\dot{\chi}+\frac{\dd V_n(\chi)}{\dd\chi}=-\lambda\,\dot{\phi}\,, \\
    &\ddot{\phi}+3H\dot{\phi}+\frac{\dd V(\phi)}{\dd\phi}=\lambda\,\dot{\chi}\,,
\label{kg2}
\end{align}
where $\lambda>0$ is the coupling constant. 

For each respective field, we define the following dimensionless quantities:
\begin{align*}
        &\Phi\equiv\phi/f, ~\quad m_\phi\equiv \Lambda^2_\phi/f, ~\quad \alpha_\phi\equiv f/M_p, \quad ~x\equiv H_0 t, \quad \text{and} \quad \mu_\phi\equiv m_\phi/H_0, \\
        &X\equiv \chi/F, \quad m_\chi\equiv \Lambda^2_\chi/F, \quad \alpha_\chi\equiv F/M_p, \quad ~x\equiv H_0 t, \quad \text{and} \quad \mu_\chi\equiv m_\chi/H_0,
    \end{align*}
These changes in variables morph \eqref{kg1} into
\begin{equation}
    X^{\prime\prime}+3\frac{p}{x}X^\prime+\alpha_\chi^{-2}\frac{\dd V_n(X)}{\dd X}=-\lambda\frac{\alpha_\phi}{\alpha_\chi}\Phi^\prime,
    \label{Dless kg1}
\end{equation}
and \eqref{kg2} into
\begin{equation}
    \Phi^{\prime\prime}+3\frac{p}{x}\Phi^\prime+\alpha^{-2}_\phi\frac{\dd V(\Phi)}{\dd\Phi}=\lambda\frac{\alpha_\chi}{\alpha_\phi}X^\prime.
    \label{Dless kg2}
\end{equation}
As usual, the primes represent derivatives with respect to the independent variable $x$.
We continue by finding the derivatives of the potentials in \eqref{p1} and \eqref{p2},
\begin{align*}
\frac{\dd V(\Phi)}{\dd\Phi} &= -\mu_\phi^{2}\alpha_\phi^{2}\sin{\Phi}, \\
\frac{\dd^2V(\Phi)}{\dd\Phi^2} &= -\mu_\phi^{2}\alpha_\phi^{2}\cos{\Phi}, \\
\frac{\dd V_n(X)}{\dd X} &= \mu_\chi^{2}\alpha_\chi^{2}n\sin{X}(1-\cos{X})^{n-1}, \\
\frac{\dd^2 V_n(X)}{\dd X^2} &= \mu_\chi^{2}\alpha_\chi^{2}n(n-1)\sin^2{X}(1-\cos{X})^{n-2}+\mu_\chi^{2}\alpha_\chi^{2}n\cos{X}(1-\cos{X})^{n-1} \\
&= \frac{\dd V_n(X)}{\dd X}\left(\frac{n+n\cos{X}-1}{\sin{X}}\right).  
\end{align*}
After an expansion of the potentials around their initial values $X_i~\text{and}~\Phi_i$, we get
\begin{equation*}
    \frac{\dd V(\Phi_i)}{\dd\Phi} \approx \frac{\dd^2 V(\Phi_i)}{\dd\Phi^2}(\tan{\Phi_i}+\delta\Phi)\,,
\end{equation*}
and
\begin{equation*}
    \frac{\dd V_n(X_i)}{\dd X} \approx \frac{\dd^2V_n(X_i)}{\dd X^2} \left(\frac{\sin{X_i}}{n+n\cos{X_i}-1}+\delta X\right).
\end{equation*}
Here, $\delta\Phi = \Phi-\Phi_i$ and $\delta X=X-X_i$. Now \eqref{Dless kg1} and \eqref{Dless kg2} become
\begin{equation*}
\delta X^{\prime\prime}+3\frac{p}{x}\delta X^\prime+\alpha_\chi^{-2}\frac{\dd^2V_n(X_i)}{\dd X^2}\left(\frac{\sin{X_i}}{n+n\cos{X_i}-1}+\delta X\right) = -\gamma_\chi\delta\Phi^\prime \,,
\end{equation*}
and
\begin{equation*}
\delta\Phi^{\prime\prime}+3\frac{p}{x}\delta\Phi^\prime+\alpha_\phi^{-2}\frac{\dd^2V(\Phi_i)}{\dd\Phi^2}(\tan{\Phi_i}+\delta\Phi) = \gamma_\phi\delta\chi^\prime\,,
\end{equation*}
where we have defined the interaction terms as
\begin{equation}
\gamma_\chi \equiv \lambda\frac{\alpha_\phi}{\alpha_\chi}\,, \quad \text{and} \quad \gamma_\phi\equiv\lambda\frac{\alpha_\chi}{\alpha_\phi}\,.    
\label{gammaint}
\end{equation}
We define the mass terms
\begin{align*}
M_\chi^2 &\equiv \alpha_\chi^{-2}\frac{\dd^2V_n(X_i)}{\dd X^2},\\
M_\phi^2 &\equiv \alpha_\phi^{-2}\frac{\dd^2V(\Phi_i)}{\dd\Phi^2},
\end{align*}
and the new variables
\begin{align*}
y_\chi(x) &=\delta X +\frac{\sin{X_i}}{n+n\cos{X_i}-1}\,, \\
y_\phi(x)&= \delta\Phi+\tan{\Phi_i}\,,    
\end{align*}
which allows us to cast the differential equations in the forms
\begin{equation}
    y_\phi^{\prime\prime}+3\frac{p}{x}y^\prime_\phi+M_\phi^2 y_\phi=\gamma_\phi y_\chi^\prime\,,
    \label{yphi}
\end{equation}
and
\begin{equation}
y_\chi^{\prime\prime}+3\frac{p}{x}y^\prime_\chi+M_\chi^2 y_\chi=-\gamma_\chi y_\phi^\prime\,.
\label{ychi}
\end{equation}
 We then define a complex-valued function $u(x) =y_\phi(x) +i\,y_\chi(x)$, where $i$ is the imaginary unit. Under our assumptions $M_\phi^2=M_\chi^2\equiv M^2,~\text{and}~\gamma_\phi=\gamma_\chi\equiv\gamma$, multiplying \eqref{ychi} by the imaginary unit $i$ and adding to \eqref{yphi} results in 
\begin{equation}
u''+\left(3\frac{p}{x}+i\gamma\right)u'+M^2u=0.
\label{ueq}
\end{equation}
The hypergeometric function of the first kind $_1F_1(A;b;z)$ satisfies~(\ref{eq:kummer}). We turn \eqref{ueq} into the form of Kummer's equation, \eqref{eq:kummer}, by letting 
\begin{equation*}
u = e^{i\kappa x}f(z(x))\,,    
\end{equation*}
with $z(x)=\xi x^n$, and where $\kappa$, $\xi$ and $n$ are constants to be determined. Then
\begin{align*}
u^\prime&=e^{i\kappa x}\frac{\dd f}{\dd z}\frac{\dd z}{\dd x}+i\kappa e^{i\kappa x}f(z)\,, \\
u^{\prime\prime}&=e^{i\kappa x}\frac{\dd^2 f}{\dd z^2}\left(\frac{\dd z}{\dd x}\right)^2+2i\kappa e^{i\kappa x}\frac{\dd f}{\dd z}\frac{\dd z}{\dd x}-\kappa^2 e^{i\kappa x}f(z).
\end{align*}
Substituting $u^\prime$ and $u^{\prime\prime}$ into \eqref{ueq} gives
\begin{align}
&\frac{\dd^2 f}{\dd z^2}\xi^2 n^2 x^{2(n-1)}+\frac{\dd f}{\dd z}\left[\xi n(n-1)x^{n-2}+\xi n x^{n-1}\left(3\frac{p}{x}+2i\kappa+i\gamma\right)\right] \nonumber \\
&\hspace{9.67cm}+\left(M^2+3\frac{p}{x}i\kappa-\gamma \kappa-\kappa^2\right)f=0 \nonumber \\
&\frac{\dd^2 f}{\dd z^2}\xi^2 n^2 x^{2(n-1)}+\frac{\dd f}{\dd z}\xi nx^{n-2}\big(n-1+3p+2i\kappa x+i\gamma x\big)+\left(M^2+3\frac{p}{x}i\kappa-\gamma \kappa-\kappa^2\right)f=0 \nonumber \\
&\frac{\dd^2 f}{\dd z^2} \xi x^n+\frac{\dd f}{\dd z}(n-1+3p+2i\kappa x+i\gamma x)+ \frac{x^{1-n}}{\xi n}\Big[(M^2-\gamma \kappa-\kappa^2)x+3ip\kappa\Big]f=0\,.
\label{eqf}
\end{align}
We extract the values for $n$, $\kappa$ and $\xi$ such that \eqref{eqf} takes the form of \eqref{eq:kummer}. We begin by observing that $M^2-\gamma \kappa-\kappa^2$ must vanish, that is,
\begin{equation*}
M^2-\gamma \kappa-\kappa^2=0 \quad\text{ or }\quad \kappa=-\frac{\gamma}{2}\mp \sqrt{M^2+\frac{\gamma^2}{4}}\equiv -\frac{\gamma}{2}\mp k\,.
\end{equation*}
We choose the negative root of $\kappa$ and by setting $n=1$, we get
\begin{align*}
\xi x\frac{\dd^2 f}{\dd z^2}+(3p-2ikx)\frac{\dd f}{\dd z}-i\frac{3p}{\xi}\left(\frac{\gamma}{2}+k\right)f&=0.
\end{align*}
Finally, by letting $\xi \equiv 2ik$, we arrive at
\begin{align}
z\frac{\dd^2 f}{\dd z^2}+(3p-z)\frac{\dd f}{\dd z}-\frac{3p}{2k}\left(\frac{\gamma}{2}+k\right)f&=0\,.
\label{coupled_solution}
\end{align}
We can identify from \eqref{coupled_solution} and \eqref{eq:kummer} that 
\begin{equation*}
b\equiv 3p\,, \quad \text{and} \quad A\equiv \frac{3p}{2k}\left(\frac{\gamma}{2}+k\right)\,.
\end{equation*}
Therefore, the general solution is
\begin{equation*}
f(z) = C_1\, {}_1F_1(A;3p;z) + C_2\, U(A;3p;z)\,,
\end{equation*}
where $U(A;3p;z)$ is the Tricomi function~(\ref{eq:tricomi}). 
Hence
\begin{equation*}
u(x)=e^{-i\left(k+\gamma/2\right)x}\Big[C_1\,{}_1F_1\!\left(A;3p;2ikx\right)+C_2\, U\!\left(A;3p;2ikx\right)\Big].
\end{equation*}
Imposing regularity at the origin eliminates the Tricomi function $U(A;3p;2ikx)$ which blows up when $x\to 0$. The function ${}_1F_1(A;3p;2ikx)\to 1$ as $x\to 0$. Therefore, the regular solution is
\begin{equation*}
u(x)=C\,e^{-i\left(k+\gamma/2\right)x}\,{}_1F_1\!\left(A;3p;2ikx\right)\,,
\end{equation*}
where $C=K_1+iK_2$.  
The two real fields are then
\begin{align}
y_\chi(x) &= \Im\!\bigl[u(x)\bigr]=\beta(X_i)+\delta X=\beta(X_i)+X-X_i\,, \\
y_\phi(x) &= \Re\!\bigl[u(x)\bigr]=\tan\Phi_i+\delta\Phi=\tan\Phi_i+\Phi-\Phi_i\,.
\end{align}
The real constants of integration $K_1$ and $K_2$ can be determined from initial conditions: at $x=0$, $X=X_i$ and $\Phi=\Phi_i$, so for $X$, we get $K_2=\beta(X_i)$ and for $\Phi$ we get $K_1=\tan\Phi_i$. Therefore, the function $u(x)$ becomes
\begin{equation*}
u(x)=\Big[\tan\Phi_i+i\beta(X_i)\Big]\,e^{-i\left(k+\gamma/2\right)x}\,{}_1F_1\!\left(A;3p;2ikx\right)\,,
\end{equation*}
and our dark matter and dark energy fields evolve as
\begin{align}
\label{eq:xfinal}
X(x)&=X_i-\beta(X_i)+\Im\!\bigl[u(x)\bigr]\,, \\
\Phi(x)&=\Phi_i-\tan\Phi_i+\Re\!\bigl[u(x)\bigr]\,.
\label{eq:phifinal}
\end{align}
where
\begin{align}
    \beta(X_i)&=\frac{\sin X_i}{n+n\cos X_i-1}\,.
    \label{beta}
\end{align}
\end{proof}

The theorem above establishes that under certain conditions the coupled dark matter--dark energy system possesses, at early times, an exact analytical solution in terms of Kummer's confluent hypergeometric function ${}_1F_1$. Having obtained closed-form expressions for the interacting fields, it is natural to investigate their physical behavior in limiting regimes. In cosmological applications, particularly important cases correspond to very weak coupling and very strong coupling, since these limits determine whether the dark sectors evolve nearly independently or as a single effective system. We therefore analyze the asymptotic behavior of the solutions in these two regimes.

\begin{figure}[t]
    \centering
    \includegraphics[width=1\linewidth]{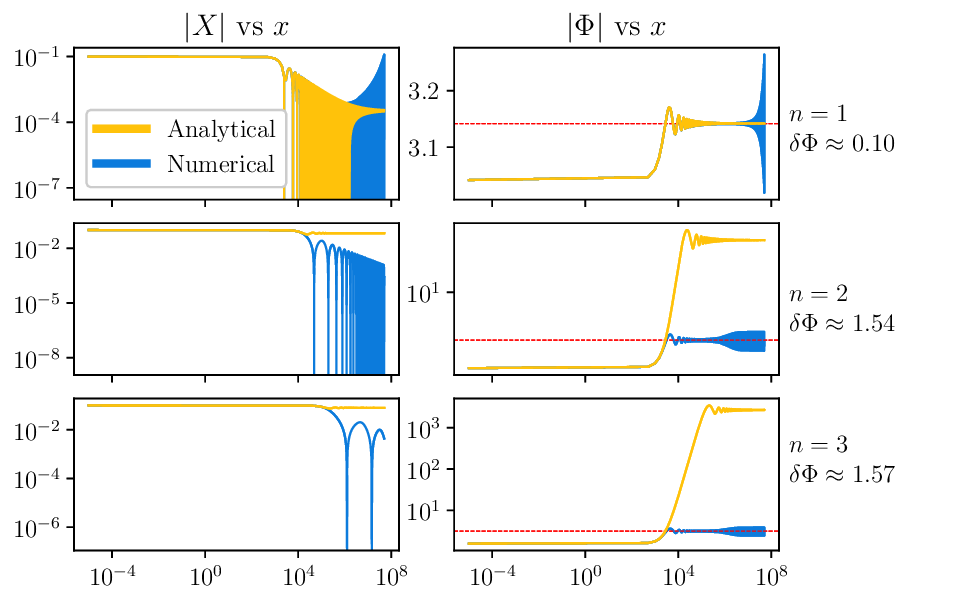}
    \caption{Evolution of the fields $X$ and $\Phi$ with the time parameter $x$ for the cases $n=1$ (top panels), $n=2$ (middle panels) and $n=3$ (bottom panels). Here, the coupling is turned off, i.e., $\lambda=0$ and the initial field value is $X_i=0.1$. The dashed red lines indicate $\Phi=\pi$.}
    \label{fig:lambda 0 Xi 0.1}
\end{figure}

\begin{figure}[t]
    \centering
    \includegraphics[width=1\linewidth]{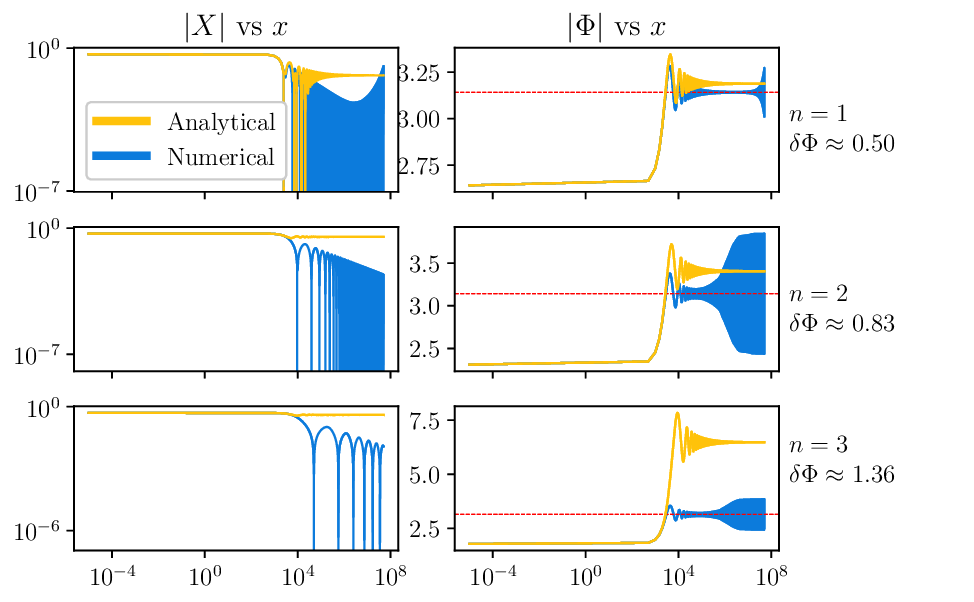}
    \caption{Same as in Fig.~\ref{fig:lambda 0 Xi 0.1} but for $X_i=0.5$. }
    \label{fig:lambda 0 Xi 0.5}
\end{figure}

\subsection{Asymptotic behavior of the fields}

We discuss the asymptotic nature of $X(x)$ and $\Phi(x)$ for $\gamma\to 0$ and $\gamma\to\infty$. 

We first consider the case of very large coupling, $\gamma\gg M$ (or $\gamma\to\infty$). We get
\begin{equation}
    u(x)=\tan\Phi_i+i\beta(X_i)\,.
    \label{ulimit}
\end{equation}

For the case of zero or very small coupling ($\gamma\to 0$ or $\gamma\ll M$), we get using~(\ref{eq:f11bessel})
\begin{equation*}
u(x)=\Big[\tan\Phi_i+i\beta(X_i)\Big]\,\Gamma\left(\frac{3p+1}{2}\right)\left(\frac{Mx}{2}\right)^{\frac{1-3p}{2}}J_{\frac{3p-1}{2}}(Mx)\,,
\end{equation*}
where $J_{\frac{3p-1}{2}}$ is the Bessel function of the first kind~(\ref{eq:besselrep}) of fractional order $(3p-1)/2$ and $\Gamma(a)$ is the gamma function~(\ref{eq:gamma}).
In this limit, $\gamma\to 0$ and for $Mx>0$, the field solutions~(\ref{eq:xfinal}) and~(\ref{eq:phifinal}) become
\begin{align}
X(x)&=X_i+\beta(X_i)\left[-1+\Gamma\left(\frac{3p+1}{2}\right)\left(\frac{Mx}{2}\right)^{\frac{1-3p}{2}}J_{\frac{3p-1}{2}}(Mx)\right]\,, \\
\Phi(x)&=\Phi_i+\tan\Phi_i\left[-1+\Gamma\left(\frac{3p+1}{2}\right)\left(\frac{Mx}{2}\right)^{\frac{1-3p}{2}}J_{\frac{3p-1}{2}}(Mx)\right]\,.
\end{align}
Note that in the limit $x\to 0$ we have
\begin{equation*}
    \Gamma\left(\frac{3p+1}{2}\right)\left(\frac{Mx}{2}\right)^{\frac{1-3p}{2}}J_{\frac{3p-1}{2}}(Mx)\longrightarrow 1+\mathcal{O}(Mx)\,,
\end{equation*}
which means that $X(0)\to X_i$ and $\Phi(0)\to\Phi_i$ as expected.

\section{Numerical analysis}

In this section, we give a numerical analysis exhibiting the evolution of the fields $X$ and $\Phi$ with cosmic time, represented by the independent variable $x$. Our aim is to show the ranges over which the analytical and numerical solutions of the KG equations agree and where deviations begin to appear. Recall that the analytical solutions were obtained by Taylor expanding the potentials $V_n(\chi)$~(\ref{p1}) and $V(\phi)$~(\ref{p2}) around small field excursions, i.e., $\delta X=X-X_i\ll X_i$ and $\delta\Phi=\Phi-\Phi_i\ll \Phi_i$. Once the fields overcome the friction term (represented by the Hubble parameter), they will begin to move away from their initial values. We expect disagreements between the analytical and numerical solutions when the fields stray far from their initial values, while the two solutions should track each other as long as $\delta X$ and $\delta\Phi$ remain small enough. The evolution of the fields $X$ and $\Phi$ in Figs.~\ref{fig:lambda 0 Xi 0.1}$-$\ref{fig:lambda 1000 Xi 0.1} show the analytical solutions of~(\ref{eq:xfinal}) and~(\ref{eq:phifinal}) and the exact numerical solutions of~(\ref{Dless kg1}) and~(\ref{Dless kg2}) for a number of cases with different values of $n$ and the interaction (coupling) strength $\lambda$. Note that since we take $\alpha_\phi=\alpha_\chi$, then from~(\ref{gammaint}), we have $\gamma\equiv\lambda$. In all of the figures, we take $p=1/2$ (for radiation domination).

\begin{figure}[t]
    \centering
    \includegraphics[width=1\linewidth]{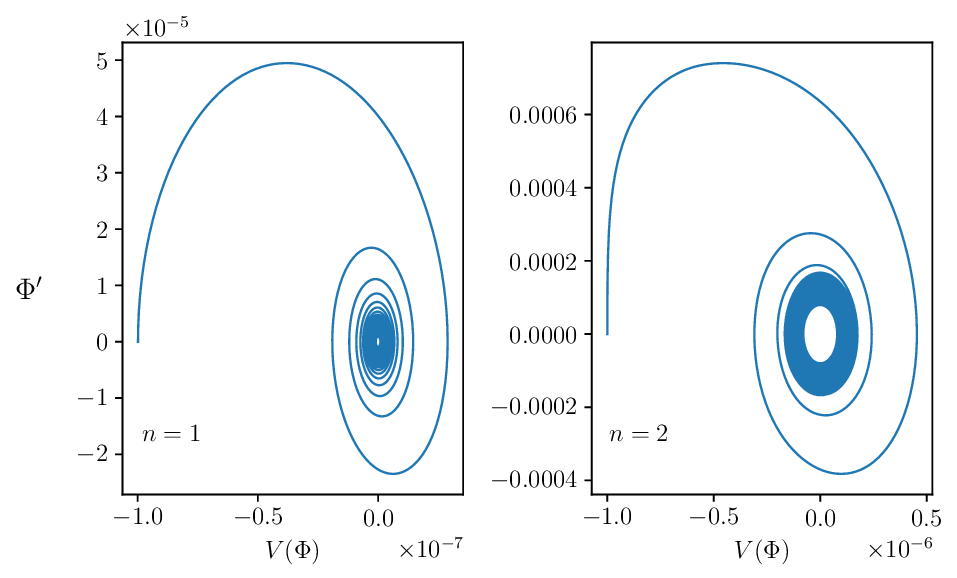}
    \caption{The phase space diagram showing the field velocity $\Phi^\prime$ versus the potential $V(\Phi)$ for the $n=1$ (left panel) and $n=2$ (right panel) cases when the interaction is switched off. The characteristic spirals seen in both cases signal the presence of a fixed point or an attractor solution for the KG equation of $\Phi$. }
    \label{fig:Phiprime vs Potential_n_1}
\end{figure}

\begin{figure}[t]
    \centering
    \includegraphics[width=1\linewidth]{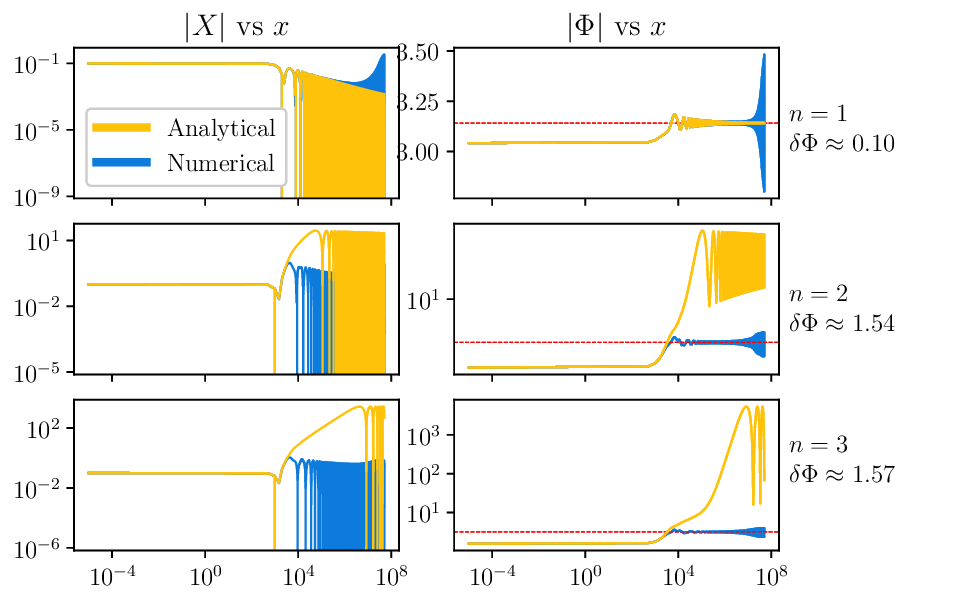}
    \caption{Same as in Fig.~\ref{fig:lambda 0 Xi 0.1} but for $\lambda=10^{-3}$.  }
    \label{fig:lambda 0.001 Xi 0.1}
\end{figure}

\begin{figure}[t]
    \centering
    \includegraphics[width=1\linewidth]{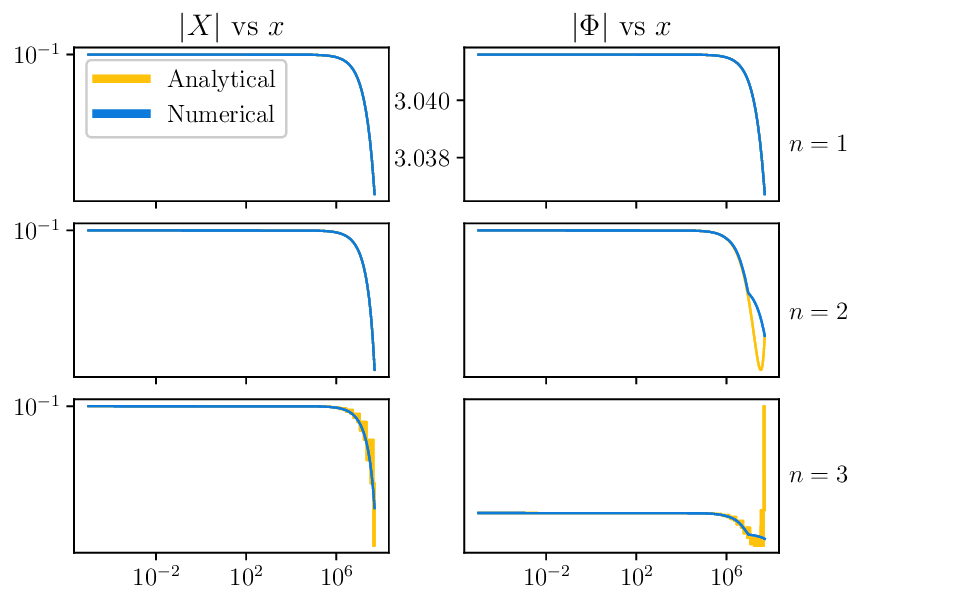}
    \caption{Same as in Fig.~\ref{fig:lambda 0 Xi 0.1} but for $\lambda=10^{3}$.}
    \label{fig:lambda 1000 Xi 0.1}
\end{figure}

We begin with the case of no interaction, i.e., $\lambda=0$ and with an initial field value $X_i=0.1$. The evolutions of the fields $X$ and $\Phi$ in Fig.~\ref{fig:lambda 0 Xi 0.1} are shown for the cases $n=1$ (top panels), $n=2$ (middle panels) and $n=3$ (bottom panels). For the $n=1$ case, the analytical and numerical solutions track each other almost perfectly over the entire time range, since $\delta\Phi\approx 0.1\ll \Phi_i\approx 3.0$. The apparent deviations seen in the late time are primarily due to the onset of rapid oscillations that the fields undergo~\cite{Urena-Lopez:2015gur}. For example, in the case of the $\Phi$ field, the oscillations in the numerical solution average out and converge to the analytical solution. For the $n=2$ and $n=3$ cases in Fig.~\ref{fig:lambda 0 Xi 0.1}, the numerical and analytical solutions show an agreement up to $x\sim 10^3$, after which they diverge away from each other, with $\delta\Phi\approx 1.54$ and $\delta\Phi\approx 1.57$, respectively. The initial field values for $n=2$ and $n=3$ are $\Phi_i\approx 1.6$ and $\Phi_i\approx 1.57$, which are clearly equal to or comparable to $\delta\Phi$, invalidating the accuracy of Taylor expansion of the potential. Although there is no coupling in this case, our assumption of equal masses, $M_\chi^2=M_\phi^2$, imposes a dependency of $\Phi_i$ on $X_i$ and $n$, which explains why different $\Phi_i$ values are obtained for the same choice of $X_i$. 

Compared to the choice of $X_i=0.1$, changing the initial field value to $X_i=0.5$ introduces a slight deviation between the analytical and numerical solutions for the $n=1$ case as seen in Fig.~\ref{fig:lambda 0 Xi 0.5}, where the field excursion has risen to $\delta\Phi\approx 0.5$. Deviations between the two solutions for $X$ and $\Phi$ are visible around $x\sim 10^3$ for all cases, $n=1$, $n=2$ and $n=3$. However, compared to the previous case, deviations in the $n=2$ and $n=3$ cases remain almost identical for the field $X$, while for $\Phi$, they haven been tamed and the effect is very stark in the $n=3$ case, where previously (for $X_i=0.1$) the deviation in $\Phi$ was almost two order of magnitude, while for the $X_i=0.5$, it has become an $\mathcal{O}(1)$ effect. The reason for this can be understood by realizing that no matter what the initial value of $\Phi$ is, the numerical solution always converges to $\pi$. The system has an attractor solution~\cite{Amendola:1999dr} (or a fixed point) which can be seen in the phase diagram depicting the field velocity $\Phi^\prime$ versus the potential $V(\Phi)$ as shown in Fig.~\ref{fig:Phiprime vs Potential_n_1} for the $n=1$ (left panel) and $n=2$ (right panel) cases. It is clear that the solution converges to $V(\Phi)=0$, where $V(\Phi)=\mu_\phi^2\alpha_\phi^2(1+\cos\Phi)$, which indeed gives $\Phi=\pi$ for $V(\Phi)=0$. So, higher $\Phi_i$ values mean smaller $\delta\Phi$, which explains why the numerical and analytical solutions do not deviate widely apart as in the case shown in Fig.~\ref{fig:lambda 0 Xi 0.1}. Here, too, the deviation is mainly driven by the value of $n$. 

We now turn our attention to the case of weak coupling, by setting $\lambda=10^{-3}$. The evolution of the fields $X$ and $\Phi$ for an initial field value $X_i=0.1$ are shown in Fig.~\ref{fig:lambda 0.001 Xi 0.1}. For $n=1$ (top panel), a very good agreement between the analytical and numerical solutions persists for both $X$ and $\Phi$, with the analytical solution showing more violent oscillations compared to the no-interaction $(\lambda=0)$ case. Since the initial value $\Phi_i$ is dependent only on $n$ and $X_i$, there is no change in the initial condition and the field-excursions for all cases remain identical to those in Fig.~\ref{fig:lambda 0 Xi 0.1}. Again, for the $n=2$ and $n=3$ cases, oscillations in the analytical solutions for $X$ and $\Phi$ become more prominent and, noticeably, deviations between the analytical and numerical solutions begin earlier than for the $\lambda=0$ case. Furthermore, the numerical solution of $\Phi$ still converges, on the average, to $\pi$, with a slight asymmetry observed around the $\Phi=\pi$ line.   

For large $\lambda$, (\ref{ueq}) reveals that the coupling $\gamma$ (which is the same as $\lambda$) contributes to the friction term giving an effective Hubble parameter $H_{\rm eff}\equiv 3p/x+i\gamma$. For zero or weak coupling $\gamma$, the term is large for small $x$ and small for large $x$, i.e., the Hubble friction dominates the early time. However, with the presence of a large $\gamma$, $H_{\rm eff}$ is dominated by $3p/x$ for small $x$ and by $\gamma$ for large $x$. Therefore, the Hubble friction does not wane as the universe expands. As a result, we expect the field excursions of $X$ and $\Phi$ away from their initial values to be very minimal. This is indeed what we see in Fig.~\ref{fig:lambda 1000 Xi 0.1}. Recall that in the large $\gamma$ limit, the solution converges to~(\ref{ulimit}) which exhibits no oscillatory behavior and is only a function of the initial field values. The analytical and numerical solutions for both $X$ and $\Phi$ exactly track each other with no noticeable deviations, except at the tail-end where numerical fluctuations exist. These fluctuations are not significant; they are enhanced by the log scale in the plots. 

In summary, the hypergeometric solution to the coupled system of KG equations remains valid as long as field-excursions away from the initial field values remain small compared to these initial values. The case $n=1$ shows the best agreement between the numerical and analytical solutions over most of the considered time range, while an agreement between the two solutions is observed over a limited $x$ range for the $n=2$ and $n=3$ cases. A full agreement between the two solutions is possible for the case of very large coupling.

\section{Conclusion}

In this work, we revisited the analytical solution of the Klein-Gordon equation with the small-field excursions and extended the analysis to a system of two coupled Klein-Gordon equations. Under the assumption that field excursions remain small relative to their initial values, we showed that the coupled system admits solutions expressed in terms of Kummer's confluent hypergeometric function.

To assess the validity of the analytical approach, we performed a full numerical integration of the coupled system and compared the numerical solutions with the corresponding analytical expressions. The comparison shows that the hypergeometric solutions remain valid as long as the field excursions from their initial values are sufficiently small. Among the cases considered, the solution with $n=1$ exhibits the best agreement with the numerical results over most of the evolution range. For $n=2$ and $n=3$, agreement is maintained over a more restricted range, reflecting the limitations of the underlying approximation. We also found that the agreement between analytical and numerical solutions improves significantly with increasing coupling strength, with nearly complete agreement attainable for the case of very large coupling.

These results demonstrate that hypergeometric functions provide a robust analytical framework for describing coupled scalar-field dynamics, yielding accurate approximations and offering valuable insight into the behavior of systems that are otherwise accessible primarily through numerical methods.

\vspace{1cm}

{\bf Acknowledgments:} 
The research of AA and TY is supported in part by the NASA CT Space Grant Prosum number P-2590.

\bibliography{references.bib}

@article{Aboubrahim:2024spa,
    author = "Aboubrahim, Amin and Nath, Pran",
    title = "{Interacting ultralight dark matter and dark energy and fits to cosmological data in a field theory approach}",
    eprint = "2406.19284",
    archivePrefix = "arXiv",
    primaryClass = "astro-ph.CO",
    doi = "10.1088/1475-7516/2024/09/076",
    journal = "JCAP",
    volume = "09",
    pages = "076",
    year = "2024"
}

@article{Aboubrahim:2024cyk,
    author = "Aboubrahim, Amin and Nath, Pran",
    title = "{Upper limits on dark energy-dark matter interaction from DESI DR2 in a field-theoretic analysis}",
    eprint = "2411.11177",
    archivePrefix = "arXiv",
    primaryClass = "astro-ph.CO",
    doi = "10.1088/1475-7516/2025/10/081",
    journal = "JCAP",
    volume = "10",
    pages = "081",
    year = "2025"
}

@article{Aboubrahim:2025usl,
    author = "Aboubrahim, Amin and Giman, Andrew H. and Nath, Pran",
    title = "{Cosmology of axion dark energy in supersymmetric models and constraints on high-scale parameters}",
    eprint = "2511.22559",
    archivePrefix = "arXiv",
    primaryClass = "astro-ph.CO",
    doi = "10.1103/3s8n-j6tq",
    journal = "Phys. Rev. D",
    volume = "114",
    number = "2",
    pages = "023501",
    year = "2026"
}

@article{Poulin:2018dzj,
    author = "Poulin, Vivian and Smith, Tristan L. and Grin, Daniel and Karwal, Tanvi and Kamionkowski, Marc",
    title = "{Cosmological implications of ultralight axionlike fields}",
    eprint = "1806.10608",
    archivePrefix = "arXiv",
    primaryClass = "astro-ph.CO",
    doi = "10.1103/PhysRevD.98.083525",
    journal = "Phys. Rev. D",
    volume = "98",
    number = "8",
    pages = "083525",
    year = "2018"
}

@article{Tsujikawa:2013fta,
    author = "Tsujikawa, Shinji",
    title = "{Quintessence: A Review}",
    eprint = "1304.1961",
    archivePrefix = "arXiv",
    primaryClass = "gr-qc",
    doi = "10.1088/0264-9381/30/21/214003",
    journal = "Class. Quant. Grav.",
    volume = "30",
    pages = "214003",
    year = "2013"
}

@article{Jr:2021rnt,
    author = "Jr., W. N. Mathews and Esrick, M. A. and Teoh, Z. Y. and Freericks, J. K.",
    title = "{A physicist's guide to the solution of Kummer's equation and confluent hypergeometric functions}",
    eprint = "2111.04852",
    archivePrefix = "arXiv",
    primaryClass = "quant-ph",
    doi = "10.5488/CMP.25.33203",
    journal = "Condensed Matter Phys.",
    volume = "25",
    number = "3",
    pages = "2022",
    year = "2022"
}

@article{Aboubrahim:2026gyp,
    author = "Aboubrahim, Amin and Nath, Pran",
    title = "{Automatic detection of fast oscillations of dark matter scalar field and updated cosmological constraints on QCDM}",
    eprint = "2608.13346",
    archivePrefix = "arXiv",
    primaryClass = "astro-ph.CO",
    month = "8",
    year = "2026"
}

@article{Rahimy:2025iyj,
    author = "Rahimy, Saba and Teixeira, Elsa M. and Zavala, Ivonne",
    title = "{Deciphering coupled scalar dark sectors}",
    eprint = "2503.01961",
    archivePrefix = "arXiv",
    primaryClass = "hep-th",
    doi = "10.1103/8hgv-l6ph",
    journal = "Phys. Rev. D",
    volume = "112",
    number = "4",
    pages = "043512",
    year = "2025"
}

@article{Teixeira:2024qmw,
    author = "Teixeira, Elsa M. and Poulot, Gaspard and van de Bruck, Carsten and Di Valentino, Eleonora and Poulin, Vivian",
    title = "{Alleviating cosmological tensions with a hybrid dark sector}",
    eprint = "2412.14139",
    archivePrefix = "arXiv",
    primaryClass = "astro-ph.CO",
    doi = "10.1103/9lf2-33zf",
    journal = "Phys. Rev. D",
    volume = "113",
    number = "2",
    pages = "023514",
    year = "2026"
}

@article{CosmoVerseNetwork:2025alb,
    author = "Di Valentino, Eleonora and others",
    collaboration = "CosmoVerse Network",
    title = "{The CosmoVerse White Paper: Addressing observational tensions in cosmology with systematics and fundamental physics}",
    eprint = "2504.01669",
    archivePrefix = "arXiv",
    primaryClass = "astro-ph.CO",
    doi = "10.1016/j.dark.2025.101965",
    journal = "Phys. Dark Univ.",
    volume = "49",
    pages = "101965",
    year = "2025"
}

@article{Abdalla:2022yfr,
    author = "Abdalla, Elcio and others",
    title = "{Cosmology intertwined: A review of the particle physics, astrophysics, and cosmology associated with the cosmological tensions and anomalies}",
    eprint = "2203.06142",
    archivePrefix = "arXiv",
    primaryClass = "astro-ph.CO",
    reportNumber = "FERMILAB-CONF-22-192-SCD",
    doi = "10.1016/j.jheap.2022.04.002",
    journal = "JHEAp",
    volume = "34",
    pages = "49--211",
    year = "2022"
}

@article{Garcia-Arroyo:2024tqq,
    author = "Garcia-Arroyo, Gabriela and Ure{\~n}a-L{\'o}pez, L. Arturo and V{\'a}zquez, J. Alberto",
    title = "{Interacting scalar fields: Dark matter and early dark energy}",
    eprint = "2402.08815",
    archivePrefix = "arXiv",
    primaryClass = "astro-ph.CO",
    doi = "10.1103/PhysRevD.110.023529",
    journal = "Phys. Rev. D",
    volume = "110",
    number = "2",
    pages = "023529",
    year = "2024"
}

@article{Urena-Lopez:2015gur,
    author = "Ure{\~n}a-L{\'o}pez, L. Arturo and Gonzalez-Morales, Alma X.",
    title = "{Towards accurate cosmological predictions for rapidly oscillating scalar fields as dark matter}",
    eprint = "1511.08195",
    archivePrefix = "arXiv",
    primaryClass = "astro-ph.CO",
    doi = "10.1088/1475-7516/2016/07/048",
    journal = "JCAP",
    volume = "07",
    pages = "048",
    year = "2016"
}

@article{Amendola:1999dr,
    author = "Amendola, Luca",
    title = "{Perturbations in a coupled scalar field cosmology}",
    eprint = "astro-ph/9906073",
    archivePrefix = "arXiv",
    doi = "10.1046/j.1365-8711.2000.03165.x",
    journal = "Mon. Not. Roy. Astron. Soc.",
    volume = "312",
    pages = "521",
    year = "2000"
}

@article{Kummer1837,
  author  = {Kummer, Ernst Eduard},
  title   = {De Generali Quadam Aequatione inter Series Hypergeometricas},
  journal = {Journal f{\"u}r die reine und angewandte Mathematik},
  volume  = {15},
  pages   = {39--83},
  year    = {1836}
}

@unpublished{parida:hal-05236085,
  TITLE = {{ON CERTAIN NEW WAYS TO FIND ASYMPTOTIC FORMULAS FOR THE BESSEL FUNCTIONS OF THE FIRST KIND}},
  AUTHOR = {Parida, Subrat},
  URL = {https://hal.science/hal-05236085},
  NOTE = {working paper or preprint},
  YEAR = {2025},
  MONTH = Sep,
  HAL_ID = {hal-05236085},
  HAL_VERSION = {v1},
}

@article{Sneddon_1961, 
     title="{L. J. Slater, Confluent Hypergeometric Functions (Cambridge University Press, 1960), 247 pp., 65s.}", 
     volume={12}, 
     DOI={10.1017/S0013091500002960}, 
     number={3}, 
     journal={Proceedings of the Edinburgh Mathematical Society}, 
     author={Sneddon, I. N.}, 
     year={1961}, 
     pages={169–169}
}

@book{zettili2009quantum,  
author    = {Nouredine Zettili},   title     = {Quantum Mechanics: Concepts and Applications},   edition   = {2},   year      = {2009},   publisher = {John Wiley \& Sons},   isbn      = {978-0-470-02679-3}}

@article{Lawi2008,
author = {Stephan Lawi},
title = {{Hermite and Laguerre Polynomials and Matrix-Valued Stochastic Processes}},
volume = {13},
journal = {Electronic Communications in Probability},
number = {none},
publisher = {Institute of Mathematical Statistics and Bernoulli Society},
pages = {67 -- 84},
year = {2008},
doi = {10.1214/ECP.v13-1353},
URL = {https://doi.org/10.1214/ECP.v13-1353}
}

@article{MUSTAPHA20102178,
title = {Generalized Laguerre expansions of multivariate probability densities with moments},
journal = {Computers \& Mathematics with Applications},
volume = {60},
number = {7},
pages = {2178-2189},
year = {2010},
issn = {0898-1221},
doi = {https://doi.org/10.1016/j.camwa.2010.08.008},
url = {https://www.sciencedirect.com/science/article/pii/S0898122110005717},
author = {Hussein Mustapha and Roussos Dimitrakopoulos}
}

\end{document}